\documentclass[12pt,reqno]{amsart}

\usepackage{multirow}
\usepackage{hhline}

\usepackage{mathtools}
\usepackage{times}
\usepackage[T1]{fontenc}
\usepackage{mathrsfs}
\usepackage{latexsym}
\usepackage[dvips]{graphics}
\usepackage[titletoc, title]{appendix}
\usepackage{amsmath,amsfonts,amsthm,amssymb,amscd}
\usepackage{color}
\usepackage{hyperref}
\usepackage{amsmath}

\usepackage{color}
\usepackage{breakurl}

\usepackage{comment}
\newcommand{\bburl}[1]{\textcolor{blue}{\url{#1}}}

\newtheorem{thm}{Theorem}[section]

\newtheorem{claim}[thm]{Claim}
\newtheorem{lem}[thm]{Lemma}
\newtheorem{prop}[thm]{Proposition}
\newtheorem{exa}[thm]{Example}

\newtheorem{defi}[thm]{Definition}
\newtheorem{rek}[thm]{Remark}

\DeclareMathOperator{\supp}{supp}
\numberwithin{equation}{section}

\begin{document}

\title{On Some Characterizations of Greedy-type Bases}

\author{Pablo M. Bern\'a}
\author{H\`ung Vi\d{\^e}t Chu}

\email{\textcolor{blue}{\href{mailto:pablo.berna@cunef.edu}{pablo.berna@cunef.edu}}}
\address{Departamento de Métodos Cuantitativos, CUNEF Universidad, 28040 Madrid, Spain}

\email{\textcolor{blue}{\href{mailto:hungchu2@illinois.edu}{hungchu2@illinois.edu}}}
\address{Department of Mathematics, University of Illinois at Urbana-Champaign, Urbana, IL 61820, USA}

\begin{abstract}
    In 1999, S. V. Konyagin and V. N. Temlyakov introduced the so-called Thresholding Greedy Algorithm. Since then,
    there have been many interesting and useful characterizations of greedy-type bases in Banach spaces. 
    In this article, we study and extend several characterizations of greedy and almost greedy bases in the literature. Along the way, we give various examples to complement our main results. Furthermore, we propose a new version of the so-called Weak Thresholding Greedy Algorithm (WTGA) and show that the convergence of this new algorithm is equivalent to the convergence of the WTGA.
\end{abstract}



\thanks{The first author was supported by Grants PID2019-105599GB-100 (Agencia Estatal de Investigación, Spain) and 20906/PI/18 from Fundación Séneca (Región de Murcia, Spain). }

\thanks{The authors are thankful to Timur Oikhberg for useful feedback on the earlier draft of this paper. The authors would like to thank the anonymous referee for helpful suggestions.}

\maketitle

\tableofcontents

\section{Introduction}
Let $\mathbb{X}$ be an infinite-dimensional real Banach space equipped with the norm $\|\cdot\|$. In this article, we focus solely on Schauder bases and the underlying field $\mathbb{R}$ though some results may hold in a more general setting. Let $\mathcal{B} = (e_n)_{n=1}^\infty$ be a semi-normalized Schauder basis of $\mathbb{X}$ with basis constant $\mathbf K_b$; that is, 
\begin{enumerate}
    \item there exists a unique dual basis $\mathcal{B}^* = (e_n^*)_{n=1}^\infty \subset \mathbb{X}^*$, where $e_n^*(e_m) = \delta_{n,m}$ for all $m, n\in\mathbb{N}$ and the functionals $(e_n^*)_{n=1}^\infty$ are called the biorthogonal functionals,  
    \item there are two positive constants $c_1, c_2$ such that $c_1 \le \|e_n\|, \|e_n^*\|\le c_2$ for every $n\in\mathbb{N}$,
    \item the \textbf{partial sum projections} of order $N$ given by 
    $$S_N(x) \ :=\ \sum_{n=1}^N e_n^*(x)e_n$$
    satisfy $\lim_{N\rightarrow \infty} S_N(x) = x$ for every $x\in \mathbb{X}$. Furthermore, $\mathbf K_b$ is the smallest positive constant such that $\|S_N(x)\|\le \mathbf K_b \|x\|$ for all $x\in \mathbb{X}$ and $N\in\mathbb{N}$. 
\end{enumerate}

Since $\mathcal{B}$ is semi-normalized, for every $x\in \mathbb{X}$, we have $\lim_{n\rightarrow\infty} |e_n^*(x)| = 0$. The support of  $x\in \mathbb{X}$ is the set $\supp (x) = \{n\in \mathbb{N}: e_n^*(x)\neq 0\}$, and the set of finitely supported vectors is denoted by $\mathbb{X}_c$.

\subsection{Classical results}
A natural way to approximate $x$ by finite sums is to consider $S_N(x)$ for $N\ge 1$. It is easy to show that 
$$\|x-S_N(x)\|\ \le\ (1+\mathbf K_b)\inf_{a_n\in\mathbb R}\left\|x-\sum_{n=1}^N a_n e_n\right\|.$$ 
Though this method of approximation is convenient, such a level of accuracy is not always satisfying. To improve the level of accuracy, S. N. Konyagin and V. N. Temlyakov \cite{KT1} introduced the \textbf{Thresholding Greedy Algorithm}, where this algorithm selects the largest coefficients (in modulus) for our approximation. In particular, for each $x$, they defined the greedy ordering $\rho: \mathbb{N}\rightarrow \mathbb{N}$ such that $\rho$ is injective, $\supp (x)\subset \rho(\mathbb{N})$, and if $j < k$, then either $|e^*_{\rho(j)}(x)| > |e^*_{\rho(k)}(x)|$ or $|e^*_{\rho(j)}(x)| = |e^*_{\rho(k)}(x)|$ and $\rho(j) < \rho(k)$. The \textbf{greedy sum} of $x$ of order $m$, denoted by $\mathcal{G}_m(x)$, is 
$$\mathcal{G}_m(x): = \sum_{n=1}^m e^*_{\rho(n)}(x)e_{\rho(n)},$$
and $\Lambda_m(x): = \{\rho(1), \ldots, \rho(m)\}$ is called the \textbf{natural order-$m$ greedy set} of $x$. 

\begin{defi}[\cite{KT1}]\normalfont
A basis $\mathcal{B}$ in a Banach space $\mathbb X$ is \textbf{greedy} if there exists a constant $C\ge 1$ such that
\begin{equation}\label{e6}\left\|x-\mathcal{G}_m(x)\right\| \ \le\ C\inf\left\{\left\|x-\sum_{n\in A}a_ne_n\right\| \,:\, |A| = m, a_n\in \mathbb{R}\right\},\forall x\in \mathbb{X}, m\in \mathbb{N}.\end{equation}
If $C$ satisfies \eqref{e6}, we say $\mathcal{B}$ is $C$-greedy. For conciseness, set $$\sigma_m(x) := \inf\left\{\left\|x-\sum_{n\in A}a_ne_n\right\| \,:\, |A| = m, a_n\in \mathbb{R}\right\}.$$
\end{defi}

Note that the natural order-$m$ greedy set is uniquely determined; however, we can also define an order-$m$ greedy set (not necessarily unique) of $x$ as any set $G\subset\mathbb{N}$ (with $|G| = m$) satisfying $$\min_{n\in G}|e_n^*(x)| \ \ge\ \max_{n\notin G}|e_n^*(x)|.$$ Clearly, $x$ may have several greedy sets of the same order when $|e_j^*(x)| = |e_k^*(x)|$ for some $j\neq k$.  The corresponding greedy sum is 
$G_m(x): = \sum_{n\in G}e_n^*(x)e_n$.

Definition \eqref{e6} of greedy bases are due to S. V. Konyagin and V. N. Temlyakov \cite{KT1}, who also gave a beautiful characterization. 

\begin{thm}[Konyagin and Temlyakov \cite{KT1}]\label{KT1}
A basis $\mathcal{B}$ of a Banach space $\mathbb{X}$ is greedy if and only if
it is unconditional and democratic.
\end{thm}

For completeness, we give the definition of unconditional and democratic bases.

\begin{defi}\normalfont A basis $\mathcal{B}$ in a Banach space $\mathbb X$ is \textbf{unconditional} if for each $x\in \mathbb{X}$, the series $\sum_{n=1}^\infty e_n^*(x)e_n$
converges unconditionally. 
\end{defi}

A notable characterization of an unconditional basis $\mathcal{B} = (e_n)_{n=1}^\infty$ is that there exists a constant $K$ such that for all $N\in\mathbb{N}$, $$\left\|\sum_{n=1}^Na_ne_n\right\|\ \le\ K\left\|\sum_{n=1}^N b_n e_n\right\|,$$
whenever $|a_n|\le |b_n|$ for all $1\le n\le N$. The least constant $K$ is called the \textbf{unconditional constant} of $\mathcal{B}$ and is denoted by $\mathbf K_u$.

Another useful characterization of unconditionality is the existence of the smallest constant $\mathbf K_{su}$ (called the \textbf{suppression-unconditional constant}) that satisfies $\|P_A(x)\|\le \mathbf K_{su}\|x\|$ for all $A\subset\mathbb{N}$, where $P_A$ is the projection operator, that is, if $A$ is a finite set,
$$P_A(x)=\sum_{n\in A}e_n^*(x)e_n.$$

\begin{defi}[\cite{KT1}]\normalfont A basis $\mathcal B$ in a Banach space is  \textbf{$C$-democratic} if there is a positive constant $C$ such that
\begin{equation}\label{e110901} \|1_A\|\ \le\ C\|1_B\|,\end{equation}
for all finite sets $A, B\subset \mathbb{N}$ with $|A| = |B|$, where for a given finite set $A\subset \mathbb{N}$,
$$1_A\ : =\ \sum_{n\in A}e_n.$$
\end{defi}
\begin{rek}\normalfont
For a given $\varepsilon = (\varepsilon_n)\in \{\pm 1\}^{\mathbb{N}}$ and a finite set $A$, we can extend $1_A$ to $1_{\varepsilon A}$ as follows:
$$1_{\varepsilon A}\ =\ \sum_{n\in A}\varepsilon_n e_n.$$
\end{rek}

S. V. Konyagin and V. N. Temlyakov \cite{KT1} also defined \textbf{quasi-greedy} bases, which satisfy $$\lim_{m\rightarrow\infty}\mathcal{G}_m(x) = x, \forall x\in\mathbb{X}.$$ Later, Wojtaszczyk \cite{W} characterized quasi-greedy bases in terms of the uniform boundedness of $(\mathcal{G}_m)_{m=1}^\infty$, which is surprising as these maps are neither linear nor continuous. In particular, a basis is quasi-greedy if and only if there exists a constant $C>0$ such that
\begin{equation}\label{e110801}\|\mathcal{G}_m(x)\|\ \le\ C\|x\|, \forall x\in\mathbb{X}, \forall m\in\mathbb{N}.\end{equation}
We shall call the smallest $C$ in \eqref{e110801} the quasi-greedy constant, denoted by $\mathbf C_q$. By a perturbation argument, \eqref{e110801} can be shown to hold for all greedy sums $G_m(x)$ with the same quasi-greedy constant (see \cite[Lemma 10.2.6]{AK}.)

One may argue that it is unfair to compare the two sides in \eqref{e6}: on the left side, $\mathcal{G}_m(x)$ is a projective approximation, while the right side allows arbitrary coefficients $a_n$. To remedy this, S. J. Dilworth, N. J. Kalton, D. Kutzarova, and V. N. Temlyakov \cite{DKKT} introduced the concept of almost greedy bases.

\begin{defi}[\cite{DKKT}]\normalfont A basis $\mathcal B$ in a Banach space $\mathbb X$ is \textbf{almost greedy} if there exists a constant $C\ge 1$ such that
\begin{equation}\label{e7}\|x-\mathcal{G}_m (x)\|\ \le\ C\tilde{\sigma}_m(x), \forall x\in \mathbb{X}, m\in \mathbb{N},\end{equation}
where $\tilde{\sigma}_m(x) = \inf\{\|x-P_A(x)\|: |A| = m\}$. If $C$ verifies \eqref{e7}, then $\mathcal{B}$ is said to be $C$-almost greedy. 
\end{defi}
It is easy to see that if a basis is greedy, then it is almost greedy. The converse is not true (see \cite[Example 10.5.4]{AK}).  In the spirit of Theorem \ref{KT1}, S. J. Dilworth et al. \cite{DKKT} showed that a basis is almost greedy if and only if it is quasi-greedy and democratic. The following theorem gives another characterization of almost greedy bases.

\begin{thm}[Dilworth, Kalton, Kutzarova, and Temlyakov \cite{DKKT}]\label{dkkt03}
Let $\mathcal B$ be a basis in a Banach space $\mathbb X$. The following are equivalent:
\begin{enumerate}
    \item [(i)] $\mathcal{B}$ is almost greedy.
    \item [(ii)] For all $\lambda>1$, there exists $C_\lambda$ such that 
    \begin{equation}\label{e110703}\|x-\mathcal{G}_{\lceil \lambda m\rceil}\| \ \le\ C_\lambda \sigma_m (x), \forall m\in \mathbb{N}, \forall x\in \mathbb{X}.\end{equation}
    \item [(iii)] For some $\lambda > 1$, there exists $C$ such that 
    \begin{equation*}\|x-\mathcal{G}_{\lceil \lambda m\rceil}(x)\|\ \le\ C \sigma_m(x), \forall m\in \mathbb{N}, \forall x\in \mathbb{X}.\end{equation*}
\end{enumerate}
\end{thm}

\subsection{Recent results and extensions}
Let $[1_A]$ denote the $1$-dimensional subspace spanned by $1_A$. P. M. Bern\'{a} and O. Blasco \cite{BB} showed that it is possible to replace $\sigma_m$ in \eqref{e6} by a simpler functional $\mathcal{D}_m$. In particular, 
$$\mathcal{D}_m(x) := \inf\left\{d(x, [1_A])\,:\, A\subset \mathbb{N}, |A| = m\right\},$$
where $d(x, [1_A]) = \inf\{\|x-y\|\,:\, y\in [1_A]\}$.

\begin{thm}[Bern\'{a} and Blasco]\label{BB1}
A basis $\mathcal{B}$ in a Banach space $\mathbb X$ is greedy if and only if there exists a constant $C$ such that
$$\|x-\mathcal{G}_m(x)\|\ \le\ C\mathcal{D}_m(x), \forall x\in \mathbb{X}, \forall m\in\mathbb{N}.$$
\end{thm}

We describe our extension. Given a function $f: \mathbb{N}\rightarrow \mathbb{R}$ and $A = \{n_1, n_2, \ldots, n_m\}\subset\mathbb{N}$, define 
$$1_{f, A}\ :=\ \sum_{j=1}^m f(j)e_{n_j}\mbox{ and } \mathcal{D}^f_m(x) \ :=\ \inf\left\{d(x, [1_{f,B}])\,:\, B\subset \mathbb{N}, |B| = m\right\}.$$

\begin{defi}\normalfont
A basis $\mathcal B$ in a Banach space is $f$-\textbf{greedy} if there is a positive constant $C_f$ (possibly dependent on $f$) such that 
\begin{equation}\label{e110902}\|x-\mathcal{G}_m(x)\|\ \le\ C_f\mathcal{D}^f_m(x), \forall x\in \mathbb{X}, \forall m\in\mathbb{N}.\end{equation}
\end{defi}

Clearly, all greedy bases satisfy \eqref{e110902}, since $\sigma_m(x)\le \mathcal{D}^f_m(x)$ for all $x\in \mathbb{X}$ and $m\in \mathbb{N}$. A natural question is for which functions $f$, does \eqref{e110902} imply greediness?
Our first result in this article gives a sufficient condition on $f$ so that \eqref{e110902} implies that $\mathcal{B}$ is greedy. (By Theorem \ref{BB1}, we know that $f\equiv 1$ works.) Furthermore, we give examples of a class of functions $f$ and bases that satisfy \eqref{e110902} but are not greedy. Roughly speaking, when the values of $|f(n)|$ are not ``bounded away" from $0$ or $\infty$ and increase/decrease fast enough, then there are bases that satisfy \eqref{e110902} but is not democratic (see Section \ref{expo}.) 

\begin{defi}\normalfont
A function $f:\mathbb{N}\rightarrow \mathbb{N}$ is \textbf{regular} if $0 < c_1 \le \inf_n |f(n)|\le \sup_n |f(n)| \le c_2 < \infty$ for some constants $c_1$ and $c_2$.
\end{defi}

\begin{thm}\label{mm1} If $f$ is regular, then a basis $\mathcal B$ in a Banach space is $f$-greedy if and only if the basis is greedy.
\end{thm}

The main difference between proving Theorems \ref{BB1} and \ref{mm1} is that while the proof of Theorem \ref{BB1} shows democracy and unconditionality of $\mathcal{B}$ independently, the proof of Theorem \ref{mm1} uses unconditionality (and norm convexity) to show democracy.

It is not clear if we can replace $\tilde{\sigma}_m(x)$ in \eqref{e7} by a simpler functional that involves $1$-dimensional subspaces as in Theorem \ref{BB1}. However, 
in the spirit of Theorem \ref{dkkt03}, we can characterize almost greedy bases by modifying the left side of \eqref{e110902}. 
\begin{thm}\label{m2} Let $f$ be a regular function. The following are equivalent:
\begin{enumerate}
    \item [(i)] $\mathcal{B}$ is almost greedy.
    \item [(ii)] For all $\lambda>1$, there exists $C_\lambda$ such that 
    \begin{equation}\label{e110704}\|x-\mathcal{G}_{\lceil \lambda m\rceil}\| \ \le\ C_\lambda \mathcal{D}_m^f(x), \forall m\in \mathbb{N}, \forall x\in \mathbb{X}.\end{equation}
    \item [(iii)] For some $\lambda > 1$, there exists $C$ such that 
    \begin{equation}\label{e110702}\|x-\mathcal{G}_{\lceil \lambda m\rceil}(x)\|\ \le\ C \mathcal{D}_m^{f}(x), \forall m\in \mathbb{N}, \forall x\in \mathbb{X}.\end{equation}
\end{enumerate}
\end{thm}

\begin{rek}\normalfont
In the proof that (iii) implies (i), we need to show that $\mathcal{B}$ is quasi-greedy and democratic. Like the proof of Theorem \ref{mm1}, we cannot show quasi-greediness and democracy independently. Instead, we first prove that \eqref{e110702} implies quasi-greediness, which will then be used to show democracy. It is worth mentioning that the same example given in Section \ref{expo}  shows that if $f$ is not regular, then \eqref{e110702} does not necessarily imply almost greediness.  
\end{rek}

The next definition is handy in due course. 

\begin{defi}\normalfont
Let $f: \mathbb{N}\rightarrow \mathbb{R}$.
A basis $\mathcal{B}$ is said to be \textbf{$f$-democratic} if there exists a constant $C$ such that 
$$\|1_{f, A}\|\ \le\ C\|1_{f, B}\|$$
for all finite sets $A, B\subset \mathbb{N}$ with $|A| = |B|$. In this case, we also say that $\mathcal{B}$ is $C$-$f$-democratic. 
\end{defi}

In the proofs of Theorems \ref{mm1} and \ref{m2}, we implicitly employ the fact that a quasi-greedy basis $\mathcal{B}$ is democratic if and only if it is $f$-democratic, given that $f$ is regular. As a complement to this observation, Section \ref{nonequivdemo} gives examples of bases that are not quasi-greedy, and the equivalence fails even when $f$ is regular.

Finally, we shift our attention to quasi-greedy bases by first discussing several characterizations that are immediate from the work of P. Wojtaszczyk, S. V. Konyagin, and V. N. Temlyakov \cite{KT, W}. Then we extend the characterization to the so-called $(A, B, t)$-greedy algorithm. For $t\in [0,1]$, a finite set $G$ is said to be a \textbf{$t$-weak greedy} set of $x$ if $$\min_{n\in G}|e_n^*(x)|\ge t\max_{n\notin G}|e_n^*(x)|.$$ Let $G(x, t)$ be the set of all $t$-weak greedy sets of $x$, $\mathcal{G}_m^t(x)$ be the set of all $t$-weak greedy sums of $x$ of order $m$, and $G_m^t(x)$ be a $t$-weak greedy sum of $x$ of order $m$. 

\begin{defi}[\cite{KT}]\normalfont
Let $\mathcal{B}$ be a basis in a Banach space $\mathbb{X}$. The basis is called \textbf{$t$-quasi-greedy} if there exists a constant $C := C(t) > 0$ such that \begin{equation}\label{etqg}\|G_m^t(x)\|\ \le\ C\|x\|, \forall m\in \mathbb{N}, \forall x\in \mathbb{X}, \forall G^t_m(x)\in \mathcal{G}_m^t(x).\end{equation}
The least constant verifying \eqref{etqg} is denoted by $\mathbf C_{q,t}$, and we say that $\mathcal{B}$ has the \textbf{$\mathbf C_{q,t}$-$t$-quasi-greedy property}. 
\end{defi}

\begin{thm}[Wojtaszczyk,  Konyagin, and Temlyakov \cite{KT, W}]\label{WKT}
Let $\mathcal{B}$ be a basis in a Banach space $\mathbb{X}$.
\begin{enumerate}
    \item[(i)] Fix $t\in (0,1]$. The basis $\mathcal{B}$ is quasi-greedy if and only if $\mathcal{B}$ has the $t$-quasi-greedy property.
    \item[(ii)] Fix $t\in (0,1]$. The basis $\mathcal{B}$ is quasi-greedy if and only if for any $x\in \mathbb{X}$ and any greedy approximation $(G_m^t(x))_{m=1}^\infty$, we have $G_m^t(x)\rightarrow x$.
    \item[(iii)] The basis $\mathcal{B}$ has the $0$-quasi-greedy property if and only if it is unconditional. 
\end{enumerate}
\end{thm}

\begin{proof}
(i) The forward implication is due to S. V. Konyagin and V. N. Temlyakov \cite[Lemma 2.3]{KT}. The backward implication follows from the fact that any greedy set is $t$-weak greedy and P. Wojtaszczyk's characterization of quasi-greediness \eqref{e110801}.

(ii) The forward implication is due to P. Wojtaszczyk \cite{W}. The backward implication is due to the fact that any greedy set is $t$-weak greedy and the definition of quasi-greediness. 

(iii) If $\mathcal{B}$ is $\mathbf K_{su}$-unconditional, then $\|P_A(x)\|\le \mathbf K_{su}\|x\|$ for all $x\in \mathbb{X}$ and $A\subset \mathbb{N}$. Hence, $\mathcal{B}$ has the $\mathbf C_{q,0}$-$0$-quasi-greedy property with $\mathbf C_{q,0}\le \mathbf K_{su}$. 

Now, assume that $\mathcal{B}$ has the $\mathbf C_{q,0}$-$0$-quasi-greedy property. Let $x\in \mathbb{X}$. Since any finite $A\subset \mathbb{N}$ is a $0$-weak greedy set for $x$, the $\mathbf C_{q,0}$-$0$-quasi-greedy property implies that $\|P_A(x)\|\le \mathbf C_{q,0}\|x\|$. Therefore, $\mathcal{B}$ is unconditional with $\mathbf K_{su}\le \mathbf C_{q,0}$. 
\end{proof}

Our goal is to extend Theorem \ref{WKT} to a new algorithm, the \textbf{$(A, B, t)$-weak greedy algorithm}. Fix $t\in (0,1]$ and sets of positive integers $A$ and $B$. For some $(a,b)\in A\times B$, a finite set $\Gamma\subset \mathbb{N}$ with $|\Gamma| = m\ge a$ is called an \textbf{$(a, b, t)$-weak greedy set} of $x\in\mathbb{X}$ if
$$|e_{s(a)}^*(x, \Gamma)| \ge t |e_{\ell(b)}^*(x, \Gamma)|,$$
where $e_{\ell(b)}^*(x, \Gamma)$ is the $b$-th largest coefficient (in modulus) over all $e_k^*(x)$ with $k\in \Gamma^c$ and $e_{s(a)}^*(x, \Gamma)$ is the $a$-th smallest coefficient (in modulus) over all $e_j^*(x)$ with $j\in \Gamma$.

Let $G^{a, b, t}_m(x): = P_{\Gamma}(x)$ denote an \textbf{$(a, b, t)$-weak greedy sum} of order $m$ of $x$, $\mathcal G^{a, b, t}_m(x)$ denote the set of all $(a,b,t)$-weak greedy sums of order $m$ of $x$, and $G(x, a, b, t)$ denote the set of all $(a, b, t)$-weak greedy sets of $x$.

\begin{defi}\normalfont
	Let $\mathcal B$ a basis in a Banach space $\mathbb{X}$. The basis is said to have the \textbf{$(A,B,t)$-quasi-greedy property} if there exists an absolute constant $C > 0$ such that for all $a\in A$ and $b\in B$, 
	\begin{equation}\label{abquasi}
		\|G^{a, b, t}_m(x)\|\ \le\ C\|x\|, \forall m\ge a, \forall x\in \mathbb{X}, \forall G^{a,b,t}_m(x)\in \mathcal{G}^{a, b, t}_m(x).
	\end{equation}
The least constant verifying \eqref{abquasi} is denoted by $\mathbf Q=\mathbf Q(A,B,t)$.
\end{defi}

The next theorem has the same spirit as Theorem \ref{WKT}. 

\begin{thm}\label{m1}
Let $\mathcal{B}$ be a basis in a Banach space $\mathbb{X}$. Fix $t\in (0,1]$ and $A, B\subset\mathbb{N}$. 
\begin{enumerate}
    \item[(i)] If $A\cup B$ is finite, then $\mathcal{B}$ is quasi-greedy if and only if $\mathcal{B}$ has the $(A, B, t)$-quasi-greedy property. 
    \item [(ii)] If $A\cup B$ is finite, then $\mathcal{B}$ is quasi-greedy if and only if for any $x\in \mathbb{X}$ and any greedy approximation $( G^{a_m,1, t}_{m}(x))_{m=\max A}^\infty$ (here $(a_m)\subset A$), we have $$G^{a_m,1, t}_{m}(x)\rightarrow x.$$
    \item[(iii)] If $A\cup B$ is infinite, then $\mathcal{B}$ has the $(A, B, t)$-quasi-greedy property if and only if it is unconditional. 
\end{enumerate}
\end{thm}

\begin{rek}\normalfont
Curious readers may ask whether, for item (ii) of Theorem \ref{m1}, the number $1$ in $G^{a_m, 1, t}_m(x)$ can be replaced by larger numbers. In general, the answer is negative. For example, let $\mathcal{B} = (e_n)_{n=1}^\infty$ be a quasi-greedy basis and set $x := e_1$. If $m > \max A$, then $\{2, 3, \ldots, m+1\}$ is an $(a_m, 2, t)$-greedy set of $x$ of order $m$ for any $(a_m)\subset A$. Set $G^{a_m, 2, t}_m(x) = P_{\{2, 3, \ldots, m+1\}}(x)$. Then $G^{a_m, 2, t}_m(x) \equiv 0$ and so, $G^{a_m, 2, t}_m(x)\not\rightarrow x$.
\end{rek}

\section{Preliminaries}

The following observation about democracy simplifies several of our proofs.
\begin{prop}\label{di}
Let $\mathcal{B}$ be a basis in a Banach space.
If there exists a constant $C$ such that $\|1_A\|\le C\|1_B\|$ for all finite and disjoint sets $A, B\subset\mathbb{N}$ with $|A| = |B|$, then 
$\mathcal{B}$ is $C^2$-democratic. 
\end{prop}

\begin{proof}
Let $E$ and $F\subset \mathbb{N}$ be finite (not necessarily disjoint) sets with $|E| = |F|$. Choose a set $G\subset\mathbb{N}$ with $\min G > \max (E\cup F)$ and $|G| = |E| = |F|$. By assumption, we have $\|1_E\|\le C\|1_G\|$ and $\|1_G\|\le C\|1_F\|$. Therefore, $\|1_E\|\le C^2\|1_F\|$, which implies $C^2$-democracy. 
\end{proof}

The next proposition given as Definition 1.1.2 in \cite{AK} will be used to show that the bases in our examples are Schauder. 

\begin{prop}\label{schauDef}
Let $(e_n)_{n=1}^\infty$ be a sequence in a Banach space $\mathbb{X}$. If there is a sequence $(e_n^*)_{n=1}^\infty$ in $\mathbb{X}^*$ such that
\begin{enumerate}
    \item [(i)] $e_k^*(e_j) = \delta_{k,j}$ for all $k, j\in \mathbb{N}$,
    \item [(ii)] $x = \sum_{n=1}^\infty e_n^*(x)e_n$ for all $x\in \mathbb{X}$.
\end{enumerate}
Then $(e_n)_{n=1}^\infty$ is a Schauder basis for $\mathbb{X}$.
\end{prop}

\begin{prop}\normalfont\label{rr1}
Let $(\mathbb{X}_n)_{n=1}^\infty$ be finite dimensional Banach spaces, each of which has a Schauder basis $\mathcal{B}_n$ with basis constant $\mathbf K_{b,n}$. Suppose that $\sup_{n} \mathbf K_{b,n} < \infty$. Then the concatenation of these $\mathcal{B}_n$'s is a Schauder basis for the sum $(\oplus_{n=1}^\infty \mathbb{X}_n)_{\ell_p}$ $(1\le p < \infty)$ and $(\oplus_{n=1}^\infty \mathbb{X}_n)_{c_0}$.
\end{prop}

\begin{proof}
Let $\mathcal{B}_n = (e_{n,j})_{j=1}^{N_n}$, where $N_n$ is the dimension of $\mathbb{X}_n$, and $\mathcal{B}_n^* = (e^*_{n,j})_{j=1}^{N_n}$. We need to show that the concatenation $\mathcal{B} = (e_{n, j})_{n\ge 1, 1\le j\le N_n}$ is a Schauder basis for $\mathbb{X} = (\oplus_{n=1}^\infty \mathbb{X}_n)_{\ell_p}$ $(1\le p < \infty)$ or $(\oplus_{n=1}^\infty \mathbb{X}_n)_{c_0}$. Choose $u\ge 1$, $1\le v\le N_u$ and consider $e^*_{u, v}: \mathbb{X}_u\rightarrow \mathbb{R}$. We can extend $e^*_{u, v}$ to $\mathbb{X}$ by setting 
$e^*_{u,v}((x_n)_{n=1}^\infty) = e^*_{u, v}(x_u)$ for all $(x_n)_{n=1}^\infty\in \mathbb{X}$ with $x_n \in \mathbb{X}_n$. Then $e^*_{u, v}\in \mathbb{X}^*$ because for all $x = (x_n)_{n=1}^\infty\in \mathbb{X}$, we have
$$|e^*_{u, v}(x)| \ =\ |e^*_{u,v}(x_u)|\ \le\ \|e^*_{u, v}\|_{\mathbb{X}^*_u}\|x_u\|_{\mathbb{X}_u}\ \le\ \|e^*_{u, v}\|_{\mathbb{X}^*_u}\|x\|_{\mathbb{X}}.$$
By Proposition \ref{schauDef}, it remains to show that for all $x = (x_n)_{n=1}^\infty\in \mathbb{X}$, 
$$x\ =\ \sum_{n=1}^\infty \sum_{m=1}^{N_n} e^*_{n,m}(x) e_{n,m}.$$
Indeed, pick $\varepsilon > 0$ and let $K := \sup \mathbf K_{b, n} < \infty$. Choose $u\in \mathbb{N}$ such that for all $v\ge u$
$$\left\|x-\sum_{n=1}^v\sum_{m=1}^{N_v}e^*_{n,m}(x) e_{n,m}\right\|_{\mathbb{X}}\ \le\ \frac{\varepsilon}{K}.$$
Therefore, for all $s \ge u+1$ and $1\le t\le N_s$, we obtain
\begin{align*}\left\|x-\sum_{n=1}^s\sum_{m=1}^t e^*_{n,m}(x) e_{n,m}\right\|_\mathbb{X}&\ \le\ K\left\|x-\sum_{n=1}^{s-1}\sum_{m=1}^{N_{s-1}} e^*_{n,m}(x) e_{n,m}\right\|_\mathbb{X}\ \le\ \varepsilon.
\end{align*}
This completes our proof. 
\end{proof}

The next lemma is useful in our proof of Theorem \ref{m2}. For its proof, see \cite[Corollary 10.2.11 and Theorem 10.2.12]{AK} and \cite{BBG}. The bounds in \cite{BBG} are tighter. 
\begin{lem}\label{gu}
Suppose that $\mathcal{B}$ is a $\mathbf C_q$-quasi-greedy basis in a Banach space $\mathbb X$. For finite sets $A, B\subset \mathbb{N}$ with $A\subset B$ and real numbers $(a_n)_{n\in A}$, we have
\begin{enumerate}
\item [(i)] 
\begin{equation}\label{b1}\left\|\sum_{n\in A}a_n e_n\right\|\ \le\ 2\mathbf C_q\max_{n\in A}|a_n|\left\|\sum_{n\in A}e_n\right\|.\end{equation}
\item [(ii)] \begin{equation}\label{b2}\min_{n\in A}| a_n|\left\|\sum_{n\in A}\varepsilon_ne_n\right\|\ \le\ 2\mathbf C_q\left\|\sum_{n\in A}a_ne_n\right\|,\end{equation}
where $\varepsilon_n  = \mbox{sign}(a_n)$.
\item [(iii)] \begin{equation}
    \label{b3}\|1_{\varepsilon A}\|\ \le \ 2\mathbf{C}_q\|1_{\eta B}\|,
\end{equation}
for all $\varepsilon, \eta\in \{\pm 1\}^{\mathbb{N}}$.
\end{enumerate}
\end{lem}

\section{Characterizations of greedy and almost greedy bases}

In this section, we assume that $f:\mathbb{N}\rightarrow \mathbb{R}$ is regular with $0 < c_1 \le \inf |f(n)| \le \sup |f(n)| \le c_2 < \infty$.

\subsection{Greedy bases - Proof of Theorem \ref{mm1}}
\begin{prop}\label{pfmm1} If $f(n)\neq 0$ for all $n$ and $\mathcal{B}$ satisfies \eqref{e110902}, then $\mathcal{B}$ is unconditional. 
\end{prop}

\begin{proof}
Let $x\in \mathbb{X}_c$ with $\supp (x) = A$, $|A| = k$, and $B\subset A$. Fix $$\alpha \ >\ 2\max_n|e_n^*(x)|/\min_{n\le k} |f(n)|.$$ Set 
$$y\ =\ x + \alpha1_{f, A\backslash B} = (P_{A\backslash B}(x) + \alpha1_{f, A\backslash B}) + P_B(x).$$
Observe that for all $n_0\in A\backslash B$, 
$$|e_{n_0}^*(y)| \ \ge \ \alpha \min_{n\le k} |f(n)| - |e_{n_0}^*(x)|  \ >\ \max_{n}|e_n^*(x)| \ \ge\ \max_{n\in B}|e_n^*(x)|.$$
Let $|A\backslash B| = m$. Then $\mathcal{G}_m(y) = P_{A\backslash B}(x) + \alpha1_{f, A\backslash B}$. By \eqref{e110902}, we have
$$\|P_B(x)\|\ =\ \|y-\mathcal{G}_m(y)\|\ \le\ C_f\mathcal{D}_m^f(y)\ \le\ C_f\|y - \alpha 1_{f, A\backslash B}\|\ =\ C_f\|x\|.$$
This shows that $\mathcal{B}$ is unconditional. 
\end{proof}

\begin{prop}\label{equivfdemo}
Let $\mathcal{B}$ be unconditional with unconditional constant $\mathbf K_u$. Then
\begin{enumerate}
    \item [(i)] if $\mathcal{B}$ is $C_{d, f}$-$f$-democratic, then $\mathcal{B}$ is $\frac{c_2}{c_1}C_{d,f}\mathbf K_u^2$-democratic,
    \item [(ii)] if $\mathcal{B}$ is $C_d$-democratic, then $\mathcal{B}$ is $\frac{c_2}{c_1}C_d\mathbf K_u^2$-$f$-democratic.
\end{enumerate}
\end{prop}

\begin{proof}
We prove (i). Let $A, B\subset \mathbb{N}$ with $|A| = |B| = m$. Write $A = \{n_1, \ldots, n_m\}$ and $B = \{\ell_1, \ldots, \ell_m\}$. 
We have
\begin{align*}
    \min_{1\le n\le m}|f(n)|\|1_{A}\|&\ =\ \left\|\sum_{j=1}^m \min_{1\le n\le m}|f(n)|e_{n_j}\right\|\ \le\ \sup_{\theta_j\in \{\pm 1\}}\left\|\sum_{j=1}^m \theta_jf(j)e_{n_j}\right\|\\
    &\ \le\ \mathbf K_u\left\|\sum_{j=1}^m f(j)e_{n_j}\right\| \ \le\ \mathbf K_uC_{d, f}\left\|\sum_{j=1}^m f(j)e_{\ell_j}\right\|\\
    &\ \le\ \mathbf K_u^2C_{d, f}\max_{1\le n\le m}|f(n)|\left\|1_B\right\|.
\end{align*}
We have shown that 
$$\|1_A\|\ \le\ \frac{c_2}{c_1}\mathbf K_u^2C_{d,f}\|1_B\|.$$

We prove (ii). Let $A, B\subset \mathbb{N}$ with $|A| = |B| = m$. We have
\begin{align*}
    \|1_{f, A}\| &\ \le\ \max_{1\le n\le m}|f(n)| \mathbf K_u\|1_A\|\ \le\ C_d\max_{1\le n\le m}|f(n)| \mathbf K_u\|1_B\|\\
    &\ \le\ C_d\frac{\max_{1\le n\le m}|f(n)|}{\min_{1\le n\le m}|f(n)|}\mathbf K^2_u\|1_{f, B}\|\ \le\ C_d\frac{c_2}{c_1}\mathbf K_u^2\|1_{f, B}\|.
\end{align*}
This completes our proof.
\end{proof}

\begin{proof}[Proof of Theorem \ref{mm1}] Let $f$ be regular. Assume that $\mathcal{B}$ satisfies \eqref{e110902}.
Due to Propositions \ref{pfmm1} and \ref{equivfdemo}, it suffices to show that $\mathcal{B}$ is $f$-democratic. 
Let $A, B\subset \mathbb{N}$ with $|A| = |B| = m$. We shall show that $\|1_{f, A}\|\le C_{d,f}\|1_{f, B}\|$ for some constant $C_{d,f}$. Choose $D\subset \mathbb{N}$ such that $\min D > \max (A\cup B)$ and $|D| = m$. For $\varepsilon > 0$, set $$x\ :=\ (c_2/c_1 + \varepsilon)1_{f, D} + 1_{f, A}.$$
Observe that for all $n\in D$, 
$$|e_n^*(x)|\ =\ (c_2/c_1 + \varepsilon)|e_n^*(1_{f,D})|\ \ge\ (c_2/c_1 + \varepsilon)c_1 = c_2 + \varepsilon c_1.$$
Hence, $\mathcal{G}_m(x) = (c_2/c_1 + \varepsilon)1_{f, D}$ and so,
\begin{align*}
    \|1_{f, A}\| \ =\ \|x - \mathcal{G}_m(x)\|\ \le\ C_f\mathcal{D}_m^f(x)\ \le\ C_f\|x- 1_{f, A}\|\ =\ C_f(c_2/c_1 + \varepsilon)\|1_{f, D}\|.
\end{align*}
Letting $\varepsilon\rightarrow 0$, we get $\|1_{f, A}\|\le \frac{c_2}{c_1}C_f\|1_{f, D}\|$. Similarly, $\|1_{f, D}\|\le \frac{c_2}{c_1}C_f\|1_{f, B}\|$. Therefore, $\|1_{f, A}\|\le (\frac{c_2}{c_1}C_f)^2\|1_{f, B}\|$, as desired. 
\end{proof}

\subsection{Almost greedy bases - Proof of Theorem \ref{m2}}

\begin{prop}
If $\mathcal{B}$ satisfies \eqref{e110702}, then 
$$\|x-G_{\lceil \lambda m\rceil}(x)\|\ \le\ C \mathcal{D}_m^{f}(x), \forall m\in \mathbb{N}, \forall x\in \mathbb{X},$$
for every greedy sum $G_{\lceil \lambda m\rceil}(x)$.
\end{prop}

\begin{proof}
Let $x\in \mathbb{X}$ and $G_{\lceil \lambda m\rceil}(x) = P_A(x)$. Let $\varepsilon > 0$. Choose a finite set $B\subset \mathbb{N}$ such that $A\subset B$ and 
$\|P_B(x) - x\| < \varepsilon$. Perturb $P_B(x)$ to obtain $y$ such that $P_A(x)$ is a strictly greedy sum of $y$, and $\|y-P_B(x)\| < \varepsilon$. By \eqref{e110702}, we have
\begin{align*}\|x - G_{\lceil \lambda m\rceil}(x)\|&\ \le\ \|x-P_B(x)\| + \|P_B(x) - y\| + \|y - P_A(x)\|\\
&\ \le\ 2\varepsilon + C\mathcal{D}_m^{f}(y).
\end{align*}
Let $z = \alpha 1_{f, D}$ for some scalar $\alpha$ and $|D| = m$. We have
$$\|y - z\|\ \le\ \|x-y\| + \|x-z\| \ <\ 2\varepsilon + \|x-z\|.$$
Hence, $\mathcal{D}_m^{f}(y) \ \le\ 2\varepsilon + \mathcal{D}_m^{f}(x)$. We obtain
$$\|x - G_{\lceil \lambda m\rceil}(x)\|\ \le\ 2(C+1)\varepsilon + C\mathcal{D}_m^{f}(x).$$
Letting $\varepsilon \rightarrow 0$, we have the desired result. 
\end{proof}

\begin{proof}[Proof of Theorem \ref{m2}]
Assume (i). By Theorem \ref{dkkt03}, we have \eqref{e110703}. Since $\sigma_m\le \mathcal{D}^f_m$, we have \eqref{e110704}. Hence, (i) implies (ii). That (ii) implies (iii) is trivial. We shall prove that (iii) implies (i). Assume (iii). Fix $x\in \mathbb{X}$. 

\underline{Quasi-greediness}: Let $G_m(x)$ be a greedy sum of $x$. Choose the non-negative integer $k$ satisfying $\lceil\lambda k\rceil \le m < \lceil \lambda (k+1)\rceil$. Then 
$$m - \lceil\lambda k\rceil \ \le\ \lceil\lambda (k+1)\rceil - \lceil\lambda k\rceil - 1 \ \le\ \lambda.$$
We have
\begin{align*}\|G_m(x)\|&\ \le\ \|x\| + \|x - G_{\lceil\lambda k\rceil}(x)\| + \|G_{\lceil\lambda k\rceil}(x) - G_m(x)\|\\
&\ \le\ \|x\| + C\mathcal{D}^f_k(x) + \lambda \mathbf k\|x\|\ \le\ (1+C+\lambda \mathbf k)\|x\|,
\end{align*}
where $\mathbf k := \sup_{n} \|e_n\|\|e_n^*\|$.
Hence, $\mathcal{B}$ is quasi-greedy with constant $\mathbf C_q\le 1+C+\lambda\mathbf k$. 

\underline{Democracy}: We prove democracy after proving the following claims. 

Claim 1: Let $A, B\subset \mathbb{N}$ with $|A| = m$, $|B| = \lceil\lambda m\rceil$, and $A\cap B = \emptyset$. For some absolute constant $C_{d, f}$, it holds that 
$$\|1_A\|\ \le\ C_{d, f}\|1_B\|.$$  

\begin{proof}For $\varepsilon > 0$, set 
$$x\ :=\ (c_2/c_1 + \varepsilon)1_{f, B} + 1_{f, A}.$$
It is easy to check that $\mathcal{G}_{\lceil\lambda m\rceil}(x) = (c_2/c_1 + \varepsilon)1_{f, B}$. Hence, 
$$\|1_{f, A}\| \ =\ \|x - \mathcal{G}_{\lceil\lambda m\rceil}(x)\| \ \le\ C \mathcal{D}_m^{f}(x)\ \le\ C\|x - 1_{f, A}\|\ =\ \left(\frac{c_2}{c_1}+\varepsilon\right)C\|1_{f, B}\|.$$
Letting $\varepsilon \rightarrow 0$ to obtain
\begin{equation}\label{eee11}\|1_{f, A}\|\ \le\ \frac{c_2}{c_1}C\|1_{f, B}\|.\end{equation}
By Lemma \ref{gu}, we have
\begin{equation}\label{eee2}\frac{c_1}{4\mathbf C_q^2}\|1_A\|\ \le \ \|1_{f, A}\|\mbox{ and } \|1_{f, B}\|\ \le\ 2\mathbf C_qc_2\|1_B\|.\end{equation}
From \eqref{eee11} and \eqref{eee2}, we obtain 
$$\|1_A\|\ \le\ 8\left(\frac{c_2}{c_1}\right)^2C\mathbf C_q^3\|1_B\|.$$
Set $C_{d, f} = 8\left(c_2/c_1\right)^2C\mathbf C_q^3$ and we are done. 
\end{proof}

Claim 2: Let $A, B\subset \mathbb{N}$ be disjoint, finite sets with $\lambda^2 +1 < |B|$ and $|A| = \lceil \lambda |B|\rceil - |B|$. It holds that $\|1_{A}\|\le \Delta_{d}\|1_{B}\|$ for some absolute constant $\Delta_d$.

\begin{proof}
Let $v = \lceil \lambda^2\rceil$ and $|B| = m > \lambda^2+1$. Partition $A$ into $v$ sets $A_1, A_2, \ldots, A_v$ such that for all $1\le j\le v$, we have
$$|A_j|\ \le\ \frac{\lceil \lambda m\rceil - m}{\lambda^2} + 1 \ <\ \frac{m}{\lambda}.$$
Hence, $\lambda |A_j| < m = |B|$. Choose $B_j\subset B$ such that $|B_j| = \lceil \lambda |A_j|\rceil$ and $\max B_j < \min (B\backslash B_j)$. We have
\begin{align*}\|1_{A}\|&\ \le \ \sum_{j=1}^v \|1_{A_j}\|\ \le\ \sum_{j=1}^v C_{d, f}\|1_{B_j}\| \mbox{ due to Claim 1}\\
&\ \le\ \mathbf K_bC_{d,f}\sum_{j=1}^v\|1_{B}\|\ =\ \mathbf K_bC_{d,f}\lceil \lambda^2\rceil\|1_{B}\|.\end{align*}
Set $\Delta_d =   \mathbf K_bC_{d,f}\lceil \lambda^2\rceil$ and we are done. 
\end{proof}

We are ready to prove democracy. Let $A, B\subset \mathbb{N}$ be disjoint sets with $|A| = |B| = m > \lambda^2+1$. Choose $D\subset\mathbb{N}$ such that $\min D > \max (A\cup B)$ and $|D| = \lceil \lambda m\rceil - m$. From Claims 1 and 2, we obtain
$$\|1_{A}\|\ \le\ C_{d,f}\|1_{D\cup B}\|\ \le\ C_{d,f}(\|1_{B}\| + \|1_{D}\|)\ \le\ C_{d,f}(\|1_{B}\| + \Delta_d\|1_{B}\|).$$
Therefore, $\|1_{ A}\|\le C_{d,f}(1+\Delta_d)\|1_{B}\|$. We have shown disjoint democracy, which implies democracy according to Proposition \ref{di}.
\end{proof}

\section{Undemocratic (thus, non-greedy) bases that satisfy \eqref{e110902}}\label{expo}

\begin{defi}\normalfont
A function $f$ is said to be \textbf{sparse} if there exists $g:\mathbb{N}\rightarrow \mathbb{R}_{>0}$ satisfying
\begin{enumerate}
    \item $\lim_{n\rightarrow\infty} g(n) = \infty$,
    \item $g(n)/n$ is decreasing or $\inf g(n)/n > 0$, and 
    \item for all nonempty, finite sets $A\subset \mathbb{N}$, we have
\begin{equation}\label{r301}\max_{n\in A}|f(n)|\ \ge \ \frac{g(|A|)}{|A|}\sum_{n\in A}|f(n)|.\end{equation}
\end{enumerate}
The function $g$ is called the \textbf{sparseness} of $f$.
\end{defi}

\begin{exa}\normalfont
We list several examples and nonexamples of sparse functions:
\begin{itemize}
    \item[(a)] The function $f\equiv 0$ is sparse with sparseness $g(n) = n$.
    \item[(b)] The function $f(n) = 2^{-n}$ is sparse with sparseness $g(n) = n/2$.
    \item[(c)] The function $f(n) = 3^n$ is sparse with sparseness $g(n) = 2n/3$.
    \item [(d)] If there exist constants $c_1, c_2$ such that $0 < c_1 \le \inf_{n}|f(n)|\le \sup_n |f(n)| \le c_2 < \infty$ (i.e., $f$ is regular), then $f$ is not sparse. 
    To see this, take a finite, nonempty set $A\subset \mathbb{N}$ with $|A| = m$. If $f$ is sparse with sparseness $g$, we have 
    $$c_2 \ \ge \ \max_{n\in A}|f(n)|\ \ge \ \frac{g(m)}{m}\sum_{n\in A}|f(n)|\ \ge\ g(m)c_1,$$
    which contradicts that $\lim_{n\rightarrow\infty} g(n) = \infty$.
    \item [(e)] The function $f(n) = n^{-2}$ is not sparse. Assume otherwise; \eqref{r301} implies that, for all $N\in \mathbb{N}$,
    $$\frac{1}{N^2} \ \ge\ \frac{g(N)}{N}\sum_{n=N}^{2N-1}\frac{1}{n^2}\ \ge\ \frac{g(N)}{N}\int_{N}^{2N} x^{-2}dx\ =\ \frac{g(N)}{2N^2},$$
    which implies that $g$ is bounded, contradicting $\lim_{n\rightarrow\infty} g(n) = \infty$.
\end{itemize}
\end{exa}

Let $f:\mathbb{N}\rightarrow \mathbb{R}$ be sparse with sparseness $g$. We now construct a Banach space with a non-greedy basis $\mathcal{B}$ that satisfies \eqref{e110902}.

We build a strictly increasing sequence $(N_k)_{k\ge 1}^\infty$ recursively. Let $N_0$ be the smallest positive integer such that $g(n) > e$ for all $n\ge N_0$, and for $k\ge 1$, let $N_k$ be a number satisfying  
$$N_k \ \ge\ g(N_{k-1})2^{N_{k-1}}\mbox{ and }g(N_k) \ >\ 3\log g(N_{k-1});$$
furthermore, if $\inf g(n)/n> c> 0$ for some $c>0$, then we also require 
$$N_k \ \ge\ \frac{3}{c}\log g(N_{k-1}).$$

For $k\ge 1$, let $\mathbb{X}_k$ be an $N_{k}$-dimensional sequence space equipped with the norm 
$$\|(x_1, \ldots, x_{N_k})\|_{\mathbb{X}_{k}} \ =\ \max\left\{\max_{i}|x_i|, \frac{\log g(N_{k-1})}{N_k}\sum_{i=1}^{N_k}|x_i|\right\}.$$
Let $\mathbb{X} = c_0(\oplus \mathbb{X}_k)$, which is the direct sum of $\mathbb{X}_k$ in the sense of $c_0$. Each vector $x$ in 
$\mathbb{X}$ is a sequence $(x_k)_{k=1}^\infty$, where $x_k\in \mathbb{X}_k$. We can also write $(x_k)_{k=1}^\infty$ as
$(x_k(i))_{\substack{k\in \mathbb{N}\\ 1\le i\le N_k}}$, where $x_k = (x_k(i))_{i=1}^{N_k}$. For each $n\in \mathbb{N}$, let 
$$Q_n(x) \ :=\ Q_n((x_k)_{k=1}^\infty) \ =\ x_n.$$
Since $\|x_n\|_{\mathbb{X}_n}\ \le\ \|x\|_{\mathbb{X}}$, we have $\|Q_n\| = 1$.  
Finally, we define blocks as a partition of $\mathbb{N}$: $M_1 = [1, N_1]$, $M_2 = [N_1+1, N_1+N_2]$, and $M_k = [\sum_{j=1}^{k-1} N_j + 1, \sum_{j=1}^k N_j]$ for all $k\ge 3$. With this notation, $Q_n(x)$ can be written as
$$Q_n(x) = P_{M_n} (x).$$
We consider the standard unit sequence, denoted by $\mathcal{B} = (e_n)_{n=1}^\infty$, where $e_n = (\delta_{k,n})_{k=1}^\infty$. By Proposition \ref{rr1}, $\mathcal{B}$ is a Schauder basis.

\begin{claim} The basis $\mathcal{B}$ is not democratic. \end{claim}
\begin{proof}
Let $A_k = M_k$ and $B_k\subset M_{k+1}$ with $|B_k| = |A_k| = |M_k|$. Then $$\|1_{A_k}\|_{\mathbb{X}} \ =\ \|1_{A_k}\|_{\mathbb{X}_k}\ =\ \log g(N_{k-1}).$$
Also, 
$$\|1_{B_k}\|_{\mathbb{X}}\ =\ \|1_{B_k}\|_{\mathbb{X}_{k+1}}\ =\ \max\left\{1, \frac{\log g(N_{k})}{N_{k+1}}N_k\right\}\ \le\ \max\left\{1, \frac{N_k\log g(N_{k})}{g(N_k)2^{N_k}}\right\} \ =\ 1,$$
for sufficiently large $k$.
Since 
$$\frac{\|1_{A_k}\|_{\mathbb{X}}}{\|1_{B_k}\|_{\mathbb{X}}}\ =\ \log g(N_{k-1}) \rightarrow \infty \mbox{ as }k\rightarrow\infty,$$
we know that $\mathcal{B}$ is not democratic. 
\end{proof}

\begin{claim} The basis $\mathcal{B}$ satisfies \eqref{e110902}; indeed, for all $x\in \mathbb{X}$ and $m\in\mathbb{N}$, it holds that
$$\|x - \mathcal{G}_m(x)\|_\mathbb{X}\le\ 2\mathcal{D}^f_m(x).$$
\end{claim}

\begin{proof}
Fix $x\in\mathbb{X}$, $m\in \mathbb{N}$, $A\subset\mathbb{N}$ with $|A| = m$, and $\alpha \in \mathbb{R}$. We shall show that 
$$\|x - \mathcal{G}_m(x)\|_{\mathbb{X}} \ \le\ 2\|x- \alpha 1_{f, A}\|_{\mathbb{X}}.$$
Let $y: = x-\mathcal{G}_m(x)$ and write $y = (y_k)_{k = 1}^\infty$, where $y_k\in \mathbb{X}_k$.
Then $$\|y\|_{\mathbb{X}}\ =\ \max_{k}\|y_k\|_{\mathbb{X}_k} \ =\ \|y_{r}\|_{\mathbb{X}_{r}}$$ for some $r$. If $\|y_r\|_{\mathbb{X}_r} = 0$, there is nothing to prove. Assume that $\|y_k\|_{\mathbb{X}_r} > 0$. We consider two cases.

Case 1: $\|y_r\|_{\mathbb{X}_r}\le 2\max_{i}|y_r(i)| = 2|y_r(j)|$ for some $1\le j\le N_r$. Since $y_r(j)$ is a coordinate we have after substracting the greedy sum of order $m$ from $x$, we know that $x$ has at least $m+1$ coordinates of magnitude at least $|y_r(j)|$. Hence, $\|x- \alpha 1_{f, A}\|_{\mathbb{X}}\ \ge\ |y_r(j)|$, since $|A| = m$. Therefore, 
$$\|y\|_{\mathbb{X}} \ =\ \|y_r\|_{\mathbb{X}_r}\ \le \ 2\|x-\alpha 1_{f,A}\|_{\mathbb{X}}.$$

Case 2: $\|y_r\|_{\mathbb{X}_r} > 2\max_{i}|y_r(i)|$. By the definition of $\|\cdot\|_{\mathbb{X}_r}$, we have
\begin{equation}\label{e111302}\|y_r\|_{\mathbb{X}_r}\ =\ \frac{\log g(N_{r-1})}{N_r}\|y_r\|_{\ell_1}.\end{equation}
If we can show that 
\begin{equation}\label{e111301}\|y_r\|_{\mathbb{X}_r}\ \le \ 2\|x_r - \alpha Q_r 1_{f, A}\|_{\mathbb{X}_r},\end{equation}
then we are done. Indeed, assume that we have \eqref{e111301}. Then 
$$\|y\|_{\mathbb{X}}\ =\ \|y_r\|_{\mathbb{X}_r}\ \le\ 2\|x_r - \alpha Q_r 1_{f, A}\|_{\mathbb{X}_r}\ \le\ 2\|Q_r\|\|x-\alpha 1_{f,A}\|_{\mathbb{X}} \ =\ 2\|x-\alpha 1_{f,A}\|_{\mathbb{X}}.$$
Assume, for a contradiction, that \eqref{e111301} does not hold. By the definition of $\|\cdot\|_{\mathbb{X}_r}$, we have
\begin{align}\label{e111303}
    \|y_r\|_{\mathbb{X}_r}&\ > \ 2\|x_r - \alpha Q_r 1_{f, A}\|_{\mathbb{X}_r}\nonumber\\
    &\ =\ 2\max\left\{\|x_r - \alpha Q_r1_{f, A}\|_{\infty}, \frac{\log g(N_{r-1})}{N_r}\|x_r - \alpha Q_r1_{f, A}\|_{\ell_1}\right\}.
\end{align}
We estimate $\frac{\log g(N_{r-1})}{N_r}\|x_r - \alpha Q_r1_{f, A}\|_{\ell_1}$. Write $\mathcal{G}_m(x) = P_{\Lambda}(x)$ and $\Lambda = \cup_{k=1}^\infty \Lambda_k$, where $\Lambda_k = \Lambda\cap M_k$. It follows that $y_r = P_{\Lambda^c_r}x_r$. Hence,
\begin{align}\label{e111304}\frac{\log g(N_{r-1})}{N_r}\|x_r - \alpha Q_r1_{f, A}\|_{\ell_1} &\ \ge\ \frac{\log g(N_{r-1})}{N_r}\|y_r - \alpha Q_rP_{\Lambda^c_r}1_{f,A}\|_{\ell_1}\nonumber\\
&\ \ge\ \frac{\log g(N_{r-1})}{N_r}\|y_r\|_{\ell_1} - \frac{\log g(N_{r-1})}{N_r}|\alpha|\|Q_rP_{\Lambda^c_r}1_{f,A}\|_{\ell_1}\nonumber\\
& \ =\  \|y_r\|_{\mathbb{X}_r} - \frac{\log g(N_{r-1})}{N_r}|\alpha|\|Q_rP_{\Lambda^c_r}1_{f,A}\|_{\ell_1}.
\end{align}
The last equality is due to \eqref{e111302}.
From \eqref{e111303} and \eqref{e111304}, we obtain
\begin{equation}\label{e111307}
    \|Q_rP_{\Lambda^c_r}1_{f,A}\|_{\ell_1}|\alpha| \ \ge\ \frac{N_r}{2\log g(N_{r-1})}\|y_r\|_{\mathbb{X}_r}.
\end{equation}
On the other hand, \eqref{e111303} gives
\begin{align*}
\|y_r\|_{\mathbb{X}_r}\ >\ 2\|x_r - \alpha Q_r1_{f, A}\|_{\infty} &\ \ge\ 2\|\alpha P_{\Lambda_r^c}Q_r1_{f, A} - y_r\|_{\infty}\\
&\ \ge\ 2|\alpha|\|P_{\Lambda_r^c}Q_r1_{f, A}\|_{\infty} - 2\|y_r\|_{\infty}.
\end{align*}
Hence, 
\begin{equation}\label{e111306}
    2|\alpha|\|P_{\Lambda_r^c}Q_r1_{f, A}\|_{\infty}  \ <\ \|y_r\|_{\mathbb{X}_r} + 2\|y_r\|_{\infty} \ \le\ 3\|y_r\|_{\mathbb{X}_r}.
\end{equation}
Since $f$ is sparse with sparseness $g$, 
\begin{equation}\label{e111305}\|P_{\Lambda_r^c}Q_r1_{f, A}\|_{\infty} \ \ge \ \begin{cases}\frac{g(N_r)}{N_r}\|P_{\Lambda_r^c}Q_r1_{f, A}\|_{\ell_1}&\mbox{ if }g(n)/n\mbox{ is decreasing},\\ c\|P_{\Lambda_r^c}Q_r1_{f, A}\|_{\ell_1}&\mbox{ if }\inf g(n)/n\ge c>0.\end{cases}\end{equation}

From \eqref{e111307}, \eqref{e111306}, and \eqref{e111305}, we obtain 
\begin{align*} 3\|y_r\|_{\mathbb{X}_r} &\ >\ 2|\alpha| \|P_{\Lambda_r^c}Q_r1_{f, A}\|_\infty\\
&\ \ge\ \begin{cases}2|\alpha|\frac{g(N_r)}{N_r}\|P_{\Lambda_r^c}Q_r1_{f, A}\|_{\ell_1}\ \ge\ \frac{g(N_r)}{\log g(N_{r-1})}\|y_r\|_{\mathbb{X}_r}&\mbox{ if }g(n)/n\mbox{ is decreasing},\\
2|\alpha|c\|P_{\Lambda_r^c}Q_r1_{f, A}\|_{\ell_1}\ \ge\ \frac{cN_r}{\log g(N_{r-1})}\|y_r\|_{\mathbb{X}_r}&\mbox{ if }\inf g(n)/n\ge c > 0.
\end{cases}
\end{align*}
which gives either
$$\frac{g(N_r)}{\log g(N_{r-1})} \ <\ 3 \mbox{ or } \frac{N_r}{\log g(N_{r-1})} \ <\ \frac{3}{c}.$$
both of which contradict our construction of $(N_k)_{k\ge 1}$.
\end{proof}

\section{When democracy and $f$-democracy are not equivalent}\label{nonequivdemo}

The goal of this section is to complement Proposition \ref{equivfdemo} by giving examples that show non-equivalence of democracy and $f$-democracy when $\mathcal{B}$ is conditional. For quasi-greedy bases, democracy and $f$-democracy are still equivalent (see Proposition \ref{qgequiv} below.) Hence, examples of bases that demonstrate the non-equivalence are, in fact, not quasi-greedy.

\begin{prop}\label{qgequiv}
Let $\mathcal{B}$ be a quasi-greedy basis of a Banach space $\mathbb{X}$ with quasi-greedy constant $\mathbf C_q$. Suppose that $f: \mathbb{N}\rightarrow\mathbb{R}$ is regular; that is, $0 < c_1\le \inf |f(n)| \le \sup |f(n)|\le c_2 < \infty$ for some constants $c_1$ and $c_2$. 
\begin{enumerate}
    \item [(i)] If $\mathcal{B}$ is $C_{d, f}$-$f$-democratic, then $\mathcal{B}$ is $\left(8\frac{c_2}{c_1}C_{d,f}\mathbf C_q^3\right)$-democratic.
    \item [(ii)] If $\mathcal{B}$ is $C_d$-democratic, then $\mathcal{B}$ is $\left(8\frac{c_2}{c_1}C_{d}\mathbf C_q^3\right)$-$f$-democratic. 
\end{enumerate}
\end{prop}

\begin{proof}
Assume that $\mathcal{B}$ is $C_{d,f}$-$f$-democratic. Let $A, B\subset\mathbb{N}$ with $|A| = |B| = m$. Then
\begin{equation}\label{eee1}\|1_{f, A}\|\ \le\ C_{d,f}\|1_{f, B}\|.\end{equation}
By Lemma \ref{gu}, we have
\begin{equation}\label{eee22}\frac{c_1}{4\mathbf C_q^2}\|1_A\|\ \le \ \|1_{f, A}\|\mbox{ and } \|1_{f, B}\|\ \le\ 2\mathbf C_qc_2\|1_B\|.\end{equation}
From \eqref{eee1} and \eqref{eee22}, we obtain 
$$\|1_A\|\ \le\ 8\frac{c_2}{c_1}C_{d,f}\mathbf C_q^3\|1_B\|.$$
Hence, $\mathcal{B}$ is $\left(8\frac{c_2}{c_1}C_{d,f}\mathbf C_q^3\right)$-democratic. 

Now, assume that $\mathcal{B}$ is $C_d$-democratic. Let $A, B\subset \mathbb{N}$ with $|A| = |B| = m$. Then
\begin{equation}\label{eee3}\|1_A\|\ \le\ C_d\|1_{B}\|.\end{equation}
By Lemma \ref{gu}, we have 
\begin{equation}\label{eee4}\|1_{f, A}\|\ \le\ 2\mathbf C_qc_2\|1_A\|\mbox{ and }c_1\|1_B\|\ \le\ 4\mathbf C_q^2\|1_{f, B}\|.\end{equation}
From \eqref{eee3} and \eqref{eee4}, we obtain
$$\|1_{f,A}\|\ \le\ 8\frac{c_2}{c_1}C_d\mathbf C_q^3\|1_{f,B}\|.$$
Hence, $\mathcal{B}$ is $\left(8\frac{c_2}{c_1}C_{d}\mathbf C_q^3\right)$-$f$-democratic.
\end{proof}

\subsection{Example of a democratic basis that is not $f$-democratic}\label{exa1}
Let $(N_k)_{k=1}^{\infty}$ be a strictly increasing sequence of positive integers. Define $\mathbb{X}_k$ to be an $N_k$-dimensional Banach space, where 
$$\|(x_1, \ldots, x_{N_k})\|_{\mathbb{X}_{k}}\ =\ \max\left\{\max_{i}|x_i|, \max_{j}\left|\sum_{i=j}^{N_k}x_i\right|\right\}.$$
Let $\mathbb{X}$ be $\ell_1(\oplus \mathbb{X}_k)$, the direct sum of $\mathbb{X}_k$ in the sense of $\ell_1$. 
We consider the standard unit sequence, denoted by $\mathcal{B} = (e_n)_{n=1}^\infty$, where $e_n = (\delta_{k,n})_{k=1}^\infty$. Each vector $x$ in 
$\mathbb{X}$ is a sequence $(x_k)_{k=1}^\infty$, where $x_k\in \mathbb{X}_k$. We can also write $(x_k)_{k=1}^\infty$ as
$(x_k(i))_{\substack{k\in \mathbb{N}\\ 1\le i\le N_k}}$, where $x_k = (x_k(i))_{i=1}^{N_k}$. By Proposition \ref{rr1}, $\mathcal{B}$ is a Schauder basis.

\begin{claim}
The standard unit basis $\mathcal{B}$ is democratic.
\end{claim}
\begin{proof}
Since for any finite set $A\subset \mathbb{N}$, we have $\|1_A\|_{\mathbb{X}} = |A|$, we know that $\mathcal{B}$ is democratic. 
\end{proof}

\begin{claim}
The standard unit basis $\mathcal{B}$ is not $f$-democratic, where $f(n) = (-1)^n$.
\end{claim}

\begin{proof}
Partition $\mathbb{N}$ into blocks: $M_1 = [1, N_1], M_2 = [N_1+1, N_2]$, and $M_k = [\sum_{i=1}^{k-1}N_i +1, N_k]$ for $k\ge 3$. Let $m\ge 3$ and choose
$$A \ =\ \left\{1, N_1+1, N_1 + N_2+1, \ldots, \sum_{n=1}^{m-1}N_{n}+1\right\},$$
and $B\subset M_k$ for some sufficiently large $k$. We have $\|1_A\|_{\mathbb{X}} = m$ and $\|1_B\|_{\mathbb{X}} = 1$. Therefore, $\mathcal{B}$ is not $f$-democratic.
\end{proof}

\subsection{Example of an $f$-democratic basis that is not democratic}
We define a norm on $c_{00}$
$$\|(x_1, x_2, x_3, \ldots)\| \ =\ \sup\left\{ \left|\sum_{i=1}^N (-1)^ix_{n_i}\right|\,:\, N\in \mathbb{N}, (n_i)_{i= 1}^N\subset \mathbb{N}, n_1 < n_2 < \cdots\right\}.$$
Let $\mathbb{X}$ be the completion of $c_{00}$ with respect to this norm.
Consider the standard unit sequence, denoted by $\mathcal{B} = (e_n)_{n=1}^\infty$, where $e_n = (\delta_{k,n})_{k=1}^\infty$.

\begin{claim}\label{sp}
The sequence $\mathcal{B}$ is a Schauder basis. 
\end{claim}
\begin{proof}
For each $n\in\mathbb{N}$, define $e_n^*: \mathbb{X}\rightarrow \mathbb{R}$ as $e_n^*((x_1, x_2, x_3, \ldots)) = x_n$. If $y = (y_1, y_2, y_3, \ldots)\in \mathbb{X}$, we have
$$\|y\|\ \ge\ \|y\|_{\infty} \ \ge\ |y_n|\ =\ |e_n^*(y)|, \forall n\in\mathbb{N}.$$
Hence, $e_n^*\in \mathbb{X}^*$. We shall show that for each $y\in \mathbb{X}$, $y = \sum_{n=1}^\infty e_n^*(y)e_n$. Let $\varepsilon > 0$. It suffices to show that there exists an $M\in \mathbb{N}$ such that 
$$\left\|\sum_{n=j}^\infty e_n^*(y)e_n\right\| \ \le\ \varepsilon, \forall j\ge M.$$
Suppose otherwise; that is, $\|\sum_{n=j}^\infty e_n^*(y)e_n\| > \varepsilon$ holds for infinitely many $j$. Since $y\in \mathbb{X}$, there exists a sequence $y_m\in c_{00}$ such that $\|y_m - y\| < \varepsilon$. Let $M = \max\mbox{supp}(y_m)$ and choose $j_0 > M$ such that $\left\|\sum_{n=j_0}^\infty e_n^*(y)e_n\right\| > \varepsilon$. By the definition of $\|\cdot\|$, we have the following contradiction
$$ \varepsilon \ >\  \|y - y_m\|\ \ge \ \left\|\sum_{n=j_0}^\infty e_n^*(y)e_n\right\| \ > \ \varepsilon.$$
This completes our proof. 
\end{proof}

\begin{claim}
The standard unit basis $\mathcal{B}$ is not democratic. 
\end{claim}
\begin{proof}
Let $m\in\mathbb{N}$. Choose $A = \{1, 2, 3, \ldots, m\}$ and $B = \{1, 3, 5, \ldots, 2m-1\}$. It is easy to check that $\|1_A\| = 1$ and $\|1_B\| = m$. Since $m$ is arbitrary, $\mathbb{X}$ is not democratic. 
\end{proof}

\begin{claim}
The standard unit basis $\mathcal{B}$ is $f$-democratic, where $f(n) = (-1)^n$.
\end{claim}
\begin{proof}
Let $m\in\mathbb{N}$ and $A, B\subset\mathbb{N}$ with $|A| = |B| = m$. It is easy to check that $\|1_A\| = \|1_B\| = m$. Hence, $\mathcal{B}$ is $1$-$f$-democratic. 
\end{proof}

Next, we give examples of unconditional bases that witness either democracy or $f$-democracy but not both. Due to Proposition \ref{qgequiv}, we know that $f$ cannot be regular. 

\subsection{Greedy basis that is not $f$-democratic}
We shall use the Schreier space, denoted by $\mathbb{S}$, as an example. The space was constructed by J. Schreier
\cite{S} as a counterexample to a question of Banach and Saks. It
has the property that the standard unit basis is
weakly null, but there is no subsequence that Ces\`{a}ro sums to 0. We describe the construction of the Schreier space. Let $\mathcal{F} = \{F\subset \mathbb{N}: \min F\ge |F|\}$.
The Schreier space is the completion of $c_{00}$ with respect to the norm
$$\|(x_1, x_2, x_3, \ldots)\|_{\mathbb{S}}\ =\ \sup_{F\in \mathcal{F}}\sum_{i\in F}|x_i|.$$
It is well-known that the standard unit sequence, denoted by $\mathcal{B}$, is an unconditional, Schauder basis of $\mathbb{S}$.

\begin{claim}
The basis $\mathcal{B}$ is $2$-democratic. 
\end{claim}

\begin{proof}
Let $A\subset \mathbb{N}$ be a finite set with $|A| = m$. By the definition of $\|\cdot\|_{\mathbb{S}}$, we have $\|1_A\|_{\mathbb{S}} \le m$. On the other hand, write $A = \{n_1, n_2, \ldots, n_m\}$. Consider
$$F\ :=\ \{n_{\lfloor m/2 \rfloor+1}, \ldots, n_m\}\ \in\ \mathcal{F}.$$
Clearly, 
$$\min F \ =\ n_{\lfloor m/2 \rfloor+1} \ \ge\ \lfloor m/2 \rfloor + 1 \ \ge\ m-\lfloor m/2 \rfloor\ =\ |F|.$$
Hence, $\|1_A\|_{\mathbb{S}} \ge m - \lfloor m/2\rfloor \ge m/2$. These bounds on $\|1_A\|_{\mathbb{S}}$ imply $2$-democracy. 
\end{proof}

\begin{claim}
The basis $\mathcal{B}$ is not $f$-democratic, where $f(n) = 1/n$.
\end{claim}
\begin{proof}
Let $m\in \mathbb{N}$ be sufficient large.
Let $A = \{1, \ldots, m\}$ and $B = \{m+1, \ldots, 2m\}$. Then 
$$\|1_{f,B}\|_{\mathbb{S}}\ =\ \sum_{n=1}^m \frac{1}{n}\ \ge \ \frac{1}{2}\log m.$$
We bound $\|1_{f, A}\|_{\mathbb{S}}$ from above. Let $F\in \mathcal{F}$. If $\min F > m$, then $\sum_{i\in F}|e_i^*(1_{f, A})| = 0$. If $2\le \min F \le m$, then 
$$\sum_{i\in F}|e_i^*(1_{f, A})|\ \le \ \sum_{n=\min F}^{2\min F-1}\frac{1}{n} \ \le\ \int_{\min F-1}^{2\min F-1}\frac{dx}{x}\ =\ \log\frac{2\min F-1}{\min F-1} \ =\ O(1).$$
If $\min F = 1$, then $F = \{1\}$, which gives $\sum_{i\in F}|e_i^*(1_{f, A})| = 1$. Hence, $\|1_{f, A}\|_{\mathbb{S}} = O(1)$. This completes our proof. 
\end{proof}

\subsection{Unconditional and $f$-democratic basis that is not democratic}
Observe that if $\sum_{n} |f(n)| < c < \infty$ and $f(1)\neq 0$, then any normalized basis is $f$-democratic. Indeed, for any finite set $A\subset \mathbb{N}$, we have 
$$\frac{|f(1)|}{\mathbf K_b}\ =\ \frac{1}{\mathbf K_b}\|f(1) e_{\min A}\|\ \le \|1_{f,A}\|\ \le\ \sum_{n} |f(n)| \ < \ c.$$
Hence, $\|1_{f, A}\|$ is bounded by positive constants independent of $A$, which implies $f$-democracy. 
On the other hand, there are many examples of an unconditional basis that is not democratic (see \cite[Example 10.4.4]{AK}.) 

Suppose that $f$ does not satisfy $\sum_{n} |f(n)| < \infty$, we shall construct an unconditional basis that is $f$-democratic but not democratic. 

Let $f(n) = n^{-1/2}$, $s_n = \sum_{i=1}^n \frac{1}{i}$, and $\Pi$ be the set of all permutations of $\mathbb{N}$. Let $d_1, d_2>0$ be such that $d_1\log n < s_n < d_2\log n$ whenever $n\ge 2$. 
Let $\mathbb{X}$ be the completion of $c_{00}$ with respect to the following norm
$$\|(x_1, x_2, x_3, \ldots)\|\ =\ \max\left\{\sup_{n}\frac{1}{\sqrt{s_n}}\sup_{\pi\in \Pi} \sum_{i=1}^n \frac{|x_{\pi(i)}|}{i^{1/2}}, \left(\sum_{i}|x_{2i}|^2\right)^{1/2}\right\}.$$
For convenience, set 
\begin{align*}
    \|(x_1, x_2, x_3, \ldots)\|_1 &\ =\ \sup_{n}\frac{1}{\sqrt{s_n}}\sup_{\pi\in \Pi} \sum_{i=1}^n \frac{|x_{\pi(i)}|}{i^{1/2}}\mbox{ and }\\ \|(x_1, x_2, x_3, \ldots)\|_2&\ =\ \left(\sum_{i}|x_{2i}|^2\right)^{1/2}.
\end{align*}
Using the same argument as in the proof of Claim \ref{sp}, we can show that $\mathcal{B} = (e_n)_{n=1}^\infty$, where $e_n = (\delta_{k,n})_{k=1}^\infty$, is a Schauder basis. Clearly, $\mathcal{B}$ is unconditional. 

\begin{claim}
The basis $\mathcal{B}$ is not democratic.  
\end{claim}
\begin{proof}
Let $m\ge 5$. Choose $A = \{2, 4, \ldots, 2m\}$ and $B = \{1, 3, \ldots, 2m-1\}$. Then
$\|1_{A}\|_2 \ =\ m^{1/2}$ and $\|1_B\|_2\ =\ 0$. For any $D\subset \mathbb{N}$ with $|D| = m$, we shall estimate
$\|1_D\|_1$:
\begin{align*}
    \|1_D\|_1 \ =\ \sup_{n\le m}\frac{1}{\sqrt{s_n}} \sum_{i=1}^n \frac{1}{i^{1/2}} \ \le\ O(1)\vee \sup_{5\le n\le m}\frac{2n^{1/2}-1}{(d_1\log n)^{1/2}}\ =\ O(1)\vee \frac{2m^{1/2}-1}{(d_1\log m)^{1/2}}.
\end{align*}
Therefore, for sufficiently large $m$, 
$$\|1_A\| \ = \ m^{1/2}, \mbox{ while } \|1_B\|\ \le\ \frac{2m^{1/2}-1}{(d_1\log m)^{1/2}}.$$
Hence, $\mathcal{B}$ is not democratic. 
\end{proof}

\begin{claim}
The basis $\mathcal{B}$ is $f$-democratic.  
\end{claim}
\begin{proof}
Let $A\subset \mathbb{N}$ with $|A| = m$. Write $A = \{n_1, n_2, \ldots, n_m\}$. We have
$$\|1_{f,A}\|_2 \ \le\ \left(\sum_{i=1}^m\frac{1}{i}\right)^{1/2} \ =\ s_m^{1/2},$$
while 
$$\|1_{f, A}\|_1 \ =\ \sup_{n\le m}\frac{1}{\sqrt{s_n}} \sum_{i=1}^n \frac{1}{i} \ =\ s_m^{1/2}.$$
Therefore, $\|1_{f, A}\| = s_m^{1/2}$ and $\mathcal{B}$ is thus, $f$-democratic. 
\end{proof}

\section{Characterization of quasi-greedy bases - Proof of Theorem \ref{m1}}
We present an useful result due to Wojtaszczyk (see Lemma \ref{l1}). For existing proofs, see \cite{W} or \cite[Lemma 2.1]{KT}. We shall improve the boundedness constant in the lemma from $C_q^4$ down to $C_q^3$.

\begin{lem}\label{l1}
Let $(e_n)_{n=1}^\infty$ be a quasi-greedy basis with constant $\mathbf C_q$. Fix $t\in (0,1]$ and
$x\in \mathbb{X}$. Let $A_1$ and $A_2$ be finite sets of positive integers such that $A_1\subset A_2$ and for all $n\in A_2$, we have $t\le |e_n^*(x)|\le 1$. Then
$$\|P_{A_1}(x)\|\ \le\ \frac{8\mathbf C_q^3}{t}\|P_{A_2}(x)\|.$$
\end{lem}

\begin{proof}
Let $\varepsilon \equiv (\text{sign}(e_n^*(x))$. Using Lemma \ref{gu}, we obtain
\begin{equation*}
    \|P_{A_1}(x)\| \ \le\ 2\mathbf C_q\|1_{A_1}\| \ \le\ 4\mathbf C_q^2 \frac{1}{t} t\|1_{\varepsilon A_2}\| \ \le\ \frac{8\mathbf C_q^3}{t} \|P_{A_2}(x)\|.
\end{equation*}
\end{proof}

\begin{lem}\label{e110802}
Let $\mathcal{B}$ be a quasi-greedy basis in a Banach space $\mathbb{X}$ with constant $\mathbf C_q$. Fix $x\in \mathbb{X}$ and $\varepsilon > 0$. 
\begin{enumerate}
\item[(a)] There exists $N\in \mathbb{N}$ such that for all $m\ge N$, 
$$\|x-G_m(x)\| \ <\ \varepsilon.$$
\item[(b)] For any sequence of non-negative real numbers $(\alpha_i)_{i\ge 1}$, if $\lim_{i\rightarrow \infty}\alpha_i = 0$, then 
$$\lim_{i\rightarrow \infty}\sum_{n: |e_n^*(x)|\ge \alpha_i} e_n^*(x)e_n  = x.$$
\item[(c)] For any sequence of non-negative real numbers $(\alpha_i)_{i\ge 1}$, if $\lim_{i\rightarrow \infty}\alpha_i = 0$, then 
$$\lim_{i\rightarrow \infty}\sum_{n: |e_n^*(x)| > \alpha_i} e_n^*(x)e_n  = x.$$
\end{enumerate}
\end{lem}

\begin{proof}
We prove (a). Choose $M\in\mathbb{N}$ such that $\|S_M (x) - x\| < \varepsilon/(2\mathbf C_q)$, where $S_M$ is the partial sum projection of order $M$. We claim that there exists $N_0\in\mathbb{N}$ such that for all $m\ge N_0$, if $P_A(x)$ is a greedy sum of order $m$, then $P_{A}(x) - S_M(x) = P_{A\backslash \{1,2,\ldots, M\}}(x)$. Indeed, let $$\alpha = \min_{\substack{1\le n\le M\\ e^*_n(x)\neq 0}}|e^*_n(x)| > 0\mbox{ and }B = \{n: |e_n^*(x)| \ge \alpha\}.$$
Then any greedy set of order at least $|B|$ must have $\{1\le n\le M: |e^*_n(x)| > 0\}$ as a subset. Set $N = N_0 = |B|$.

Let $m\ge N$ and consider a greedy sum $G_m(x) := P_A(x)$. We have
\begin{align*}
    \|x-G_m(x)\|&\ \le\ \|x-S_M(x)\| + \|S_M(x) - P_A(x)\|\\
    &\ =\ \|x-S_M(x)\| + \|P_{A\backslash\{1,2,\ldots, M\}}(x)\|\\
    &\ <\ \varepsilon/2 + \mathbf C_q\|x-S_M(x)\|\\
    & \mbox{ since }P_{A\backslash\{1,2, \ldots M\}}(x) \mbox{ is a greedy sum of } x- S_M(x)\\
    & \ <\ \varepsilon.
\end{align*}

We prove (b). Without loss of generality, we can assume that $\alpha_i > 0$ for all $i$. 

Case 1: If $x$ is finitely supported, then set $\alpha = \min_{n\in \supp (x)} |e^*_n(x)|$. There exists $i_0$ such that for all $i\ge i_0$, we have $\alpha_i < \alpha$. Hence, for all $i\ge i_0$, $\sum_{n: |e^*_n(x)|\ge \alpha_i} e^*_n(x)e_n  = x$. 

Case 2: $x$ is infinitely supported. Let $\varepsilon > 0$. Let $N$ be as in part (a). Choose $\delta > 0$ such that $|\{n: |e^*_n(x)| \ge \delta\}| > N$. Let $i_0$ be such that for all $i\ge i_0$, we have $\alpha_i < \delta$. Then for each $i\ge i_0$, 
$\sum_{n: |e^*_n(x)|\ge \alpha_i}e^*_n(x)e_n$ is a greedy sum of order at least $N$. By part (a), 
$$\left\|x-\sum_{n: |e^*_n(x)|\ge \alpha_i}e^*_n(x)e_n\right\| < \varepsilon.$$

The proof of (c) is similar to the proof of (b).
\end{proof}

\begin{proof}[Proof of Theorem \ref{m1} item (i)]
We prove the forward implication. Let $\mathbf C_q$ be the quasi-greedy constant. 
Let $x\in \mathbb{X}$ and $G^{a, b, t}_m(x) := P_G(x)$, where $|G| = m$, be an $(a, b, t)$-weak greedy sum of $x$. Pick $E\subset G$ with $|E| < a$ and $F\subset \mathbb{N}\backslash G$ with $|F| < b$ such that $P_{G\backslash E}(x)$ is a $t$-weak greedy set of $x-P_{E\cup F}(x)$. By Theorem \ref{WKT}, there exists a constant $C(t)$ verifying
$$\|x-P_{G\backslash E}(x)\|\ \le\ C(t)\|x-P_{E\cup F}(x)\|.$$
Therefore,
\begin{align*}
    \|x-P_G(x)\|&\ \le\ \|x-P_{G\backslash E}(x)\| + \|P_E(x)\|\\
    &\ \le\ C(t)\|x-P_{E\cup F}(x)\| + |E|\mathbf k\|x\|, \mbox{ where }\mathbf k = \sup_n \|e_n\|\|e_n^*\|\\
    &\ \le\ C(t)(\|x\| + (|E|+|F|)\mathbf k\|x\|) + (a-1)\mathbf k\|x\|\\
    &\ \le\ C(t)(\|x\| + (a+b-2)\mathbf k\|x\|) + (a-1)\mathbf k\|x\|\\
    &\ \le\ ((\max A-1) \cdot \mathbf k + C(t)(1+(\max A + \max B-2)\mathbf k))\|x\|.
\end{align*}
Hence, $\mathcal{B}$ has the $(A, B, t)$-quasi-greedy property.

Now, we prove the backward implication. Assume that $\mathcal{B}$ has the $(A, B, t)$-quasi-greedy property with constant $\mathbf Q$. Let $x\in \mathbb{X}$. We shall show that $\mathcal{G}_m(x)\rightarrow x$. If $x$ is finitely supported, then there is nothing to prove. Assume that $x$ is infinitely supported. Let $\mathcal{G}_m(x) = P_G (x)$ be a greedy sum of $x$ of order $m$ with $m>\max A$. Then $G\in G(x, a, b, t)$ for all $(a, b)\in A\times B$. By the $(A, B, t)$-quasi-greedy property, $\|\mathcal{G}_m(x)\|\le \mathbf Q\|x\|$, implying quasi-greediness.\end{proof}

\begin{proof}[Proof of Theorem \ref{m1} item (ii)] Assume that $\mathcal{B}$ is quasi-greedy with constant $\mathbf C_q$. Consider a sequence of $(a_m, 1, t)$-weak greedy sums $(G_m^{a_m, 1, t}: = P_{\Lambda_m}(x))_{m=\max A}^\infty$, where $|\Lambda_m|\rightarrow \infty$. Let $\alpha_m: = \max\{|e_n^*(x)|: n\notin \Lambda_m\}$. It must be that $\lim_{m\rightarrow\infty} \alpha_m = 0$; otherwise, there exists an $\varepsilon > 0$ such that $\alpha_m\ge \varepsilon$ infinitely often, which contradicts that $\lim_{n\rightarrow \infty}|e_n^*(x)|= 0$. 

For each $m$, define 
\begin{align*}
    B_m^1 &\ =\ \{n\, :\, |e_n^*(x)| > \alpha_m\}\mbox{ and }B_m^2 \ =\ \{n\, :\, |e_n^*(x)| \ge t\alpha_m\}, \\
    \Lambda_m^1 &\ =\ B_m^1 \backslash L_{\Lambda_m}, \mbox{ where } L_{\Lambda_m} = \{n\notin \Lambda_m\,:\, |e_n^*(x)| > \alpha_m\},\\
    \Lambda_m^2&\ =\ B_m^2 \cup K_{\Lambda_m}, \mbox{ where } K_{\Lambda_m} = \{n\in \Lambda_m\, :\, |e_n^*(x)| < t\alpha_m\}.
\end{align*}
Since $\alpha_m = \max\{|e_n^*(x)|: n\notin \Lambda_m\}$, we know that $L_{\Lambda_m} = \emptyset$. Also, $|K_{\Lambda_m}| < \max A$. Write
$$P_{\Lambda_m} (x)\ =\ P_{B_m^1}(x) + P_{(\Lambda_m\cap (B_m^2\backslash B_m^1))\cup K_{\Lambda_m}}(x).$$
By Lemma \ref{e110802}, we know that $P_{B_m^1}(x)\rightarrow x$. It remains to show that 
$$P_{(\Lambda_m\cap (B_m^2\backslash B_m^1))\cup K_{\Lambda_m}}(x)\rightarrow 0.$$
Let $\mathbf k := \sup_{n}\|e_n\|\|e_n^*\|$.
By Lemma \ref{l1} and the fact that $|K_{\Lambda_m}| < \max A$, we have  
\begin{align*}\|P_{(\Lambda_m\cap (B_m^2\backslash B_m^1))\cup K_{\Lambda_m}}(x)\|&\ \le\ \|P_{\Lambda_m\cap (B_m^2\backslash B_m^1)}(x)\| + \|P_{K_{\Lambda_m}}(x)\|\\
&\ \le\  8\mathbf C_q^3\frac{1}{t}\|P_{B_m^2\backslash B_m^1}(x)\| + \mathbf k(\max A-1)\|x-P_{B_m^2}(x)\|,
\end{align*}
which approaches $0$ due to Lemma \ref{e110802}. This completes our proof of the forward implication.

We prove the backward implication. Let $x\in \mathbb{X}$. By assumption,
$ G^{a_m,1, t}_{m}(x)\rightarrow x$
for any greedy approximation $(G^{a_m, 1, t}_m (x))_{m=1}^\infty$. Let $(\mathcal{G}_r(x) := P_{\Lambda_r}x)_{r= 1}^\infty$ be the greedy approximation of $x$. Whenever $r > \max A$, we have $\Lambda_r\in G(a_m, 1, t)$ for any $a_m\in A$. Therefore, $\mathcal{G}_r(x)\rightarrow x$, as desired. 
\end{proof}

\begin{proof}[Proof of Theorem \ref{m1} item (iii)] If $(e_n)_{n=1}^\infty$ is unconditional, then it has the $(A, B, t)$-quasi-greedy property for any $A, B\subset \mathbb{N}$.
We prove the forward implication. Assume that $A\cup B$ is infinite, and $(e_n)_{n=1}^\infty$ has the $(A, B, t)$-quasi-greedy property with constant $\mathbf Q$. Suppose, for a contradiction, that $(e_n)_{n=1}^\infty$ is conditional. Then there exists an $x\in \mathbb{X}$ and a sequence of finite sets of positive integers $(F_j)_{j=1}^\infty$ such that $\|P_{F_j}(x)\| \ge 2^j\|x\|$ and $\lim_{j\rightarrow\infty} |F_j| = \infty$.

Case 1: The set $A$ is infinite. Let $n_1 < n_2 < n_3 < \cdots $ be a sequence in $A$ such that $\lim_{j\rightarrow\infty} n_j = \infty$ and $n_j > |F_j|$. Choose $N > \max F_j$ sufficiently large such that $y:= S_N(x)$ satisfies $\|y\|/\|x\|\le 2$. Let $m_j\in \{k\le N: |e^*_k(y)| = \max_{n}|e^*_n(y)|\}$. Define
$$G_j: = \{m_j\}\cup F_j\cup A_j,$$
where $\min A_j > N$ and $|G_j| = |A_j| + |F_j\cup \{m_j\}| = n_j$. By definition, $G_j\in G(y, n_j, b, t)$ for all $b\in B$. Due to $(A, B, t)$-quasi-greediness, we have $\|P_{G_j}y\| \le  \mathbf Q\|y\|$. Therefore,
\begin{align*}
    2^j\|x\| &\ \le\ \|P_{F_j}(x)\|\ =\ \|P_{F_j}(y)\|\ \le \ (2\mathbf K_b+1)\|P_{F_j\cup \{m_j\}}(y)\|\\
    &\ = \  (2\mathbf K_b+1)\|P_{G_j}(y)\| \ \le\ (2\mathbf K_b+1)\mathbf Q\|y\| \ \le\ 2(2\mathbf K_b+1)\mathbf Q\|x\|.
\end{align*}
Letting $j\rightarrow\infty$, we have a contradiction. 

Case 2: The set $B$ is infinite and $\min A = s$. Let $n_1 < n_2 < n_3 < \cdots$ be a sequence in $B$. Pick $j_0$ such that for all $j\ge j_0$, $|F_j| > s$. Fix $p\ge j_0$ and choose $N > \max F_j$ sufficiently large such that $y:= S_N(x)$ satisfies $\|y\|/\|x\|\le 2$. Choose $n_k$ such that $N<n_k$. Then
$F_p\in G(y, s, n_k, t)$, since $s\le |F_p|$ and 
$|e_{\ell(n_k)}^*(y, F_p)| = 0$. By $(A,B, t)$-quasi-greediness, 
we have
\begin{align*}
    2^p\|x\| \ \le\ \|P_{F_p}(x)\| \ =\ \|P_{F_p} (y)\|\ \le\ \mathbf Q\|y\|\ \le \ 2\mathbf Q\|x\|. 
\end{align*}
Letting $p\rightarrow\infty$, we have a contradiction. 
This completes our proof. 
\end{proof}

\section{Future research}
We would like to discuss two open questions for future studies:
\begin{enumerate}
    \item Can we fully characterize functions $f$ that make \eqref{e110902} equivalent to greediness. Theorem \ref{mm1} shows that all regular functions do so, while Section \ref{expo} shows that sparse functions do not. Unfortunately, the regular functions and sparse functions do not cover all possible functions from $\mathbb{N}$ to $\mathbb{R}$. An example of functions that belong to neither class is $f(n) = 1/n^p$ for any fixed $p> 0$. 
    \item Dilworth, Kalton, and D. Kutzarova \cite{DKK} introduced the concept of semi-greedy bases and showed that semi-greediness and almost greediness are equivalent in Banach spaces with finite cotype. Bern\'{a} \cite{B2} strengthened the result by showing the equivalence for general Banach spaces. As a result, our Theorem \ref{m2} also characterizes semi-greediness. However, can we characterize semi-greediness using $\mathscr{C}\mathcal{G}_m(x)$ and $\mathcal{D}^f_m(x)$? Here $\mathscr{C}\mathcal{G}_m(x)$ is the so-called \textbf{Thresholding Chebyshev greedy sum} (see \cite[page 73]{DKK}.) More specifically, is it possible to put some conditions on $f$ such that 
    $$\|x-\mathscr{C}\mathcal{G}_m(x)\|\ \le\ C_f\mathcal{D}_m^f(x)$$
    implies semi-greediness? 
\end{enumerate}

\ \\
\end{document}